\documentclass[12pt]{amsart}
\usepackage{enumerate}
\usepackage{amsmath, amssymb, amsthm, amsfonts}
\usepackage{tikz}
\usepackage{geometry}
\usepackage{hyperref}
\usepackage{bbm} 
\usepackage{url}

\newtheorem{theorem}{Theorem}
\newtheorem{proposition}[theorem]{Proposition}
\newtheorem{corollary}[theorem]{Corollary}
\newtheorem{lemma}[theorem]{Lemma}

\newtheorem*{claim*}{Claim}

\theoremstyle{definition}

\newtheorem*{remark*}{Remark}

\def\IR{\mathbb{R}}

\def\eps{\varepsilon}

\newcommand{\defeq}{\mathrel{\vcenter{\baselineskip0.5ex \lineskiplimit0pt
                     \hbox{\scriptsize.}\hbox{\scriptsize.}}}%
                     =}

\begin{document}

\title{Acute sets of exponentially optimal size}

\author[Gerencs\'{e}r]{Bal\'{a}zs Gerencs\'{e}r}
\address{MTA Alfr\'ed R\'enyi Institute of Mathematics 
H-1053 Budapest, Re\'altanoda utca 13-15;
and E\"otv\"os Lor\'and University, Department of Probability and Statistics 
H-1117 Budapest, P\'azm\'any P\'eter s\'et\'any 1/c}
\email{gerencser.balazs@renyi.mta.hu}

\author[Harangi]{Viktor Harangi}
\address{MTA Alfr\'ed R\'enyi Institute of Mathematics 
H-1053 Budapest, Re\'altanoda utca 13-15}
\email{harangi@renyi.hu}

\thanks{The first author was supported by 
NKFIH (National Research, Development and Innovation Office) grant PD 121107. 
The second author was supported by 
``MTA R\'enyi Lend\"ulet V\'eletlen Spektrum Kutat\'ocsoport''.}




%
\begin{abstract}
We present a simple construction of an acute set of size $2^{d-1}+1$ in $\IR^d$ 
for any dimension $d$. That is, we explicitly give $2^{d-1}+1$ points 
in the $d$-dimensional Euclidean space with the property that 
any three points form an acute triangle. It is known that the maximal number of 
such points is less than $2^d$. Our result significantly improves upon a recent 
construction, due to Dmitriy Zakharov, with size of order $\varphi^d$ 
where $\varphi = (1+\sqrt{5})/2 \approx 1.618$ is the golden ratio. 
\end{abstract}

\maketitle


\section{Introduction}
Around 1950 Erd\H os conjectured that given more than $2^d$ points in $\IR^d$ 
there are three of them determining an obtuse angle. 
In 1962 Danzer and Gr\"unbaum proved this conjecture \cite{danzer/grunbaum:1962} 
(their proof can also be found in \cite{proofs_from_the_book}). 

In other words, if we want to find as many points as possible 
with all angles being at most $\pi/2$, then we cannot do better than $2^d$ points 
in dimension $d$. The vertices of the $d$-dimensional hypercube show that 
$2^d$ points exist with this property. However, in the hypercube 
a huge number of the angles are actually equal to $\pi/2$. 
A natural question arises:  what is the maximal number of points 
if we want all angles to be acute, that is, strictly less than $\pi/2$? 
A set of such points will be called an \textit{acute set}. 

The exclusion of right angles seemed to decrease 
the maximal number of points dramatically: 
Danzer and Gr\"unbaum could only find $2d-1$ points, 
and they conjectured that this is the best possible. 
However, this was only proved for $d=2, 3$. 
(For the non-trivial case $d=3$ 
see e.g.~\cite{grunbaum:1963}.) 

Later Erd\H os and F\"uredi used the probabilistic method \cite{erdos/furedi:1983} 
(choosing random vertices of the hypercube) 
to prove the existence of exponentially large acute sets of size
\begin{equation*}
\frac{1}{2} \left(\frac{2}{\sqrt{3}} \right)^d > 0.5 \cdot 1.154^d .
\end{equation*}
In \cite{harangi:2011} this construction was generalized and 
the improved bound $c \cdot 1.2^d$ was obtained. 
Recently Dmitriy Zakharov vastly surpassed the random approach 
by an explicit recursive construction \cite{zakharov:2017} 
providing $F_{d+2} > 1.618^d$ points 
where $F_n$ denotes the Fibonacci sequence. 

The next surprise came when a mathematics enthusiast from Ukraine 
(who wished to remain anonymous) came up with numerical examples of 
a $4$-dimensional acute set of size $9$ and 
a $5$-dimensional acute set of size $17$. 
(See \url{http://dxdy.ru/post1222167.html#p1222167} and \url{http://dxdy.ru/post1231694.html#p1231694} for these examples.) 
Previously, the best known lower bounds were 
$8$ in dimension $4$ and $13$ in dimension $5$, 
see \cite{harangi:2011,zakharov:2017}. 
His idea was to start from the vertices of a $(d-1)$-dimensional hypercube and 
slightly modify the coordinates. Using only one extra dimension 
he could turn the vertex set into an acute set. 
Moreover, one extra point could be easily added. 

Inspired by these examples we managed to make the same essential idea work 
in any dimension $d$. 
\begin{theorem} \label{thm:main}
There exist $2^{d-1}+1$ points in $\IR^d$ such that 
any three of them form an acute triangle. 
\end{theorem}

By this we achieve the optimal exponential rate $2$. Furthermore, 
the best known lower and upper bounds ($2^{d-1}+1$ and $2^d-1$, respectively) 
are now within a factor $2$. 

A related notion is \emph{strictly antipodal sets}. 
A set of points in $\IR^d$ is called strictly antipodal if 
for any two points $x$ and $y$ of the set there are parallel hyperplanes 
$H_x$ and $H_y$ passing through $x$ and $y$ (respectively) 
such that all other points in the set lie strictly between $H_x$ and $H_y$. 
It is easy to see that every acute set is automatically strictly antipodal. 
(We can choose $H_x$ and $H_y$ orthogonal to the line $xy$.) 
Therefore our result readily implies the following lower bound 
on the maximal cardinality of strictly antipodal sets. 
\begin{corollary}
There exists a strictly antipodal set in $\IR^d$ 
of cardinality $2^{d-1}+1$. 
\end{corollary}
The best earlier lower bound was $3^{\lfloor d/2 \rfloor - 1} - 1$ 
due to Barvinok, Lee and Novik \cite{barvinok:2013}. 
Note that the same upper bound ($2^d-1$) holds 
for strictly antipodal sets as well \cite{danzer/grunbaum:1962}. 
Hence our result implies that the optimal exponential rate 
for strictly antipodal sets is also $2$. 

\subsection*{Sketch of the construction}
Let $X$ denote a $(d-1)$-dimensional hypercube in $\IR^d$, 
and let $u$ be a unit vector orthogonal to $X$. 
The idea is to slightly perturb $X$ in a way that all the right angles in $X$ become acute. 
First we take a vertex $x$ and move it a bit closer to the center of $X$, that is, 
we choose a small $a>0$ and shift $x$ towards the center by distance $a$. 
How do the right angles (involving $x$) change?  
It is easy to see that if $x$ is the middle point of the angle, 
then the angle becomes obtuse. 
However, if it is a non-middle point, the angle will be acute. 
We can get rid of the new obtuse angles by further translating $x$ 
(this time in the orthogonal direction) with $bu$ for some small $b>0$. 
If $a$ and $b$ are appropriately coupled, 
then all angles involving $x$ become acute. 

We want to proceed similarly for all other vertices. 
We want to avoid, however, disturbing the acute angles that 
we have created so far. To this end, we will move the subsequent vertices 
by a much smaller magnitude. So we take another vertex $x_2$, 
shift it towards the center by distance $a_2$ and then translate it with $b_2u$, 
where $a_2$ and $b_2$ are appropriately coupled and much smaller than $a$ and $b$. 
If we continue this way (taking smaller and smaller pairs $a_i,b_i$), 
then all angles will be acute in the end. 
Furthermore, one more point can be easily added to this acute set: 
translate the center of the hypercube by $c u$ for some large enough $c$. 

In the rest of the paper we will make the above argument precise.


\section{Construction of the acute set}

This section contains the proof of Theorem \ref{thm:main}. 
First we give $2^{d-1}$ points by perturbing the vertices 
of a hypercube, then add an extra point. 

\subsection{Perturbation of the hypercube}

Our construction is based on the following lemma.
\begin{lemma}
Given a $(d-1)$-dimensional hypercube in $\IR^d$ and $\eps >0$ 
one can move a vertex of the hypercube by distance at most $\eps$ 
in a way that any angle determined by this point and 
two other vertices of the hypercube is acute.

Formally, let $X \subset \IR^d$ denote the vertex set of 
a $(d-1)$-dimensional hypercube and let $x \in X$ be an arbitrary vertex. 
Then for all $\eps >0$ there exists $x' \in \IR^d$ such that 
$|x-x'| \leq \eps$ and the angles $\angle x'yz$ and $\angle yx'z$ are acute 
for any distinct $y,z \in X \setminus \{x\}$. 
\end{lemma}
\begin{proof}
We may assume that 
\begin{equation*}
X = \underbrace{ \{0,1\} \times \cdots \times \{0,1\} }_{d-1} \times \{0\} 
\subset \IR^d , 
\end{equation*}
that is, we consider the set $X$ of those points for which 
each of the first $d-1$ coordinates is $0$ or $1$ and the last coordinate is $0$. 

We also assume that the vertex we want to move is $x = (0, \ldots, 0, 0) \in X$. 
We claim that the point 
\begin{equation*}
x' = ( \underbrace{ a, \ldots, a }_{d-1} , b ) \in \IR^d 
\end{equation*}
satisfies the required properties for appropriately chosen $0<a<1$ and $b$. 

We need to check that all angles formed by $x'$ and 
two distinct vertices $y,z \in X \setminus \{x\}$ are acute.

\noindent\textbf{Case 1:} $x'$ is not the middle point. 
To show that the angle $\angle x'yz$ is acute we need to prove that 
the inner product $\left\langle x'-y,z-y \right\rangle$ is positive. 
We can write this inner product as a coordinate-wise sum. 
Since $y$ and $z$ are distinct, 
there is at least one coordinate where they are different 
(one is $0$, the other is $1$). 
The contribution of such a coordinate to the inner product is either $a$, or $1-a$, 
which is positive given that $0<a<1$. 
In the other coordinates the contribution is clearly $0$. 

\noindent\textbf{Case 2:} $x'$ is the middle point. 
This time we need that the inner product 
$\left\langle y-x',z-x' \right\rangle$ is positive. 
The contribution of each of the first $d-1$ coordinates 
is one of the following: $-a(1-a)$, $a^2$, $(1-a)^2$. 
It follows that their total contribution is at least $-(d-1)a(1-a) > -(d-1)a$. 
As for the last coordinate, its contribution to the inner product is $b^2$. 
Therefore 
\begin{equation*} 
\left\langle y-x',z-x' \right\rangle > b^2 - (d-1)a .
\end{equation*}

In conclusion, all required angles are acute provided that 
\begin{equation*} 
0<a<1 \mbox{ and } b^2 \geq (d-1)a.
\end{equation*}
These conditions can be easily satisfied along with 
$|x-x'|^2 = (d-1)a^2 + b^2 \leq \eps^2 $. 
\end{proof}
By repeatedly applying the above lemma we get the following. 
\begin{proposition} \label{prop}
Given a $(d-1)$-dimensional hypercube in $\IR^d$ and $\delta >0$ 
one can move each vertex of the hypercube by distance at most $\delta$ 
such that the resulting $2^{d-1}$ points form an acute set. 
\end{proposition}
\begin{proof}
Let $N \defeq 2^{d-1}$ and let $x_1, x_2, \ldots, x_N \in X$ be an enumeration of 
the vertices of the hypercube (in an arbitrary order). 
We will apply the lemma to each $x_i$ 
(with different $\eps_i$ that we will specify later) 
and obtain new points $x'_i$. 
We start with $x_1$ and apply the lemma using $\eps_1 = \delta$. 

For any acute triangle there clearly exists a positive $\eps$ 
with the property that if each vertex of the triangle is moved 
by distance at most $\eps$, then the triangle is still acute. 
Given finitely many acute triangles we can choose a positive $\eps$ 
such that all these triangles have this property (for this common $\eps$). 

Now suppose that we have already obtained 
$x'_1, \ldots, x'_i$ for some $1 \leq i < N$ 
and let $S_i$ be the set of all triangles 
that are acute and whose vertices are among 
$x'_1, \ldots, x'_i, x_{i+1}, \ldots, x_N$. 
We choose $0<\eps_{i+1} \leq \eps_i$ in a way that 
if we move each vertex of a triangle in $S_i$ 
by distance at most $\eps_{i+1}$, then we still get an acute triangle. 

We claim that in the end the points $x'_1, \ldots, x'_N$ will form an acute set. 
We need to show that the triangle $x'_i x'_j x'_k$ is acute for any $i<j<k$. 
Since $x_j$ and $x_k$ are vertices of the hypercube, 
the triangle $x_i' x_j x_k$ is acute according to the lemma. 
Therefore this triangle is in the set $S_i$. 
Since $|x'_j - x_j| \leq \eps_j \leq \eps_{i+1}$ and 
$|x'_k - x_k| \leq \eps_k \leq \eps_{i+1}$, 
it follows that $x'_i x'_j x'_k$ is also an acute triangle. 
\end{proof}

\subsection{The cherry on the cake}

Finally, we add one more point to the acute set. 
Let $X = \{x_1, \ldots, x_N\} = \{0,1\} \times \cdots \times \{0,1\} \times \{0\}$ be 
the $(d-1)$-dimensional unit hypercube and let $\{x'_1, \ldots, x'_N\}$ be 
the acute set obtained using Proposition \ref{prop} with some small $\delta$. 
We claim that if we add $x_0 = (1/2,\ldots,1/2,c)$ to this set, 
then we will still have an acute set provided that 
$c > \sqrt{d-1}/2$ and $\delta$ is sufficiently small. 

To see this, we first note that 
the distance $| x_0 - x_i | = \sqrt{(d-1)/4+c^2}$ is the same for each $i$. 
Therefore every triangle $x_0 x_i x_j$ is isosceles, 
and consequently the angles at $x_i$ and $x_j$ are automatically acute. 
As for the angle at $x_0$, it is acute if and only if 
$ |x_i - x_j| < \sqrt{2} | x_0 - x_i | $. 
Since $|x_i - x_j|$ is at most $\sqrt{d-1}$ for any two vertices of the unit hypercube, 
this is always satisfied provided that $c > \sqrt{d-1}/2$. 
Now by choosing $\delta$ to be sufficiently small 
the triangles $x_0 x'_i x'_j$ can be arbitrarily close to the triangles $x_0 x_i x_j$, 
and hence they are acute as well. 

%
%
%
%
%
%
%
%

\end{document}